\documentclass[10pt]{article}

\usepackage{amssymb,amsfonts,amsmath}
\usepackage{graphicx}
\usepackage{url,color,bbm}
\usepackage{epstopdf}
\usepackage{subfigure}

\def\be{\begin{equation}}
\def\ee{\end{equation}}
\def\rar{\rightarrow}
\newtheorem{theorem}{Theorem}

\newtheorem{example}[theorem]{Example}
\newenvironment{proof}[1][Proof]{\textbf{#1.} }{\ \rule{0.5em}{0.5em}}
\def\C{\mathbb{C}}
\begin{document}

\begin{center}

\textbf{SPECTRAL COMPLEXITY OF DIRECTED GRAPHS AND APPLICATION TO STRUCTURAL DECOMPOSITION}

\medskip

\emph{Igor Mezi\'c$^{\S 1,2}$, Vladimir A. Fonoberov$^3$, Maria Fonoberova$^2$, Tuhin Sahai$^4$}

\medskip

$^\S$Corresponding author

$^1$ Center for Control, Dynamical Systems and Computation, University of California - Santa Barbara, Santa
Barbara, CA 93106, USA, tel: 1-805-893-7603, fax: 1-805-893-8651, e-mail: mezic@engineering.ucsb.edu

$^2$ Aimdyn, Inc., 1919 State St., Ste. 207, Santa Barbara, CA 93101, USA

$^3$ Bruker Nano, 112 Robin Hill Rd, Goleta, CA 93117, USA

$^4$ United Technologies Research Center, 2855 Telegraph Ave, Suite 410, Berkeley, CA 94115, USA

\end{center}
\begin{abstract}
We introduce a new measure of complexity (called \textit{spectral complexity}) for directed graphs. We start with splitting of the directed graph into its recurrent and non-recurrent parts. We define the spectral complexity metric
 in terms of the spectrum of the recurrence matrix (associated with the reccurent part of the graph) and the Wasserstein
distance. We show that the total complexity of the graph can then be defined in terms of the spectral
complexity, complexities of individual components and edge weights. The essential property of the spectral complexity metric
is that it accounts for directed cycles in the graph. In engineered and software systems, such cycles give rise to sub-system interdependencies and increase
risk for unintended consequences through positive feedback loops, instabilities, and infinite execution loops in software.  In addition, we present a
structural decomposition technique that identifies such cycles using a spectral technique. We show that this
decomposition complements the well-known spectral decomposition analysis based on the Fiedler vector. We provide several examples of computation
of spectral and total complexities, including the demonstration that the complexity increases monotonically with
the average degree of a random graph. We also provide an example of spectral complexity computation for the
architecture of a realistic fixed wing aircraft system. 
\end{abstract}


\section{Introduction}
Given that complex engineering systems are constructed by composing various subsystems and components that interact
with one another, it is common practice in modern engineering design to consider the directed interconnectivity
graph as a representation of the underlying system~\cite{dsm}. Thus, the question of inferring complexity of a given system
from the resulting graph arises naturally; the idea being that higher complexity graphs imply higher complexity of
system design and testing procedures~\cite{npscomplexity}. System complexity is particularly important in the context of complex aerospace systems and leads to frequent budget overruns and project delays~\cite{npscomplexity,f35}. Thus, early identification of complexity levels can enable early intervention and system redesign to mitigate risk.

A graph can be analyzed using either combinatorial graph-theoretic methods or by  matrix representations such as the adjacency matrix. In the latter case, algebraic methods for analysis are available. In particular, the spectrum of the matrix associated with an undirected graph can be related to its structural properties \cite{Cvetkovicetal:1979, chung_book}. Previously, the graph spectrum has been used to compute properties such as clusters~\cite{Tutorial,wave-clustering1} and  isomorphisms~\cite{isospectral}. Unfortunately, such relationships are not readily available in the case of directed graphs that arise frequently in typical engineering applications (and in various social network settings) due to the directionality of flow information or energy.

In directed graph theory, a common source of complexity is the existence of directed cycles in the graph.  This led
Thomas J. McCabe in 1976 to measure the complexity of a computer program \cite{mccabe,mccabe89}, using
the so-called \textit{cyclomatic complexity} - which counts the number of linearly independent cycles in the
program. A good survey on software system complexity metrics can be found in \cite{navlakha, shepperd}. We argue that these cycles are particularly important in the context of engineering systems. In particular, they are important drivers of complexity. For example, these cycles can give rise to positive feedback loops~\cite{astrom2010feedback} which lead to system instabilities. Cycles in engineering systems also make design and analysis challenging from a simulation convergence perspective~\cite{moshagen2017convergence, klus2011efficient}.

Inspired by the above argument, we develop a class of complexity metrics based on the algebraic properties of a matrix that
represents the underlying directed graph. Although there exists extensive literature on graph complexity measures of information-theoretic and energy type
\cite{DehmerandMowshowitz:2011,Dehmeretal:2014}, such measures can
either directly or indirectly be related to the moduli of eigenvalues of the underlying graph matrices. Our
approach is based on ideas that are fundamentally different from the underlying concept present in the above works. Namely, we start with the postulate that the complexity of a system should be a measure of the distance from the
least complex system of the same size. We assume that the least complex system is the one where every component is
isolated, not interacting with any other component (thus lacks any interdependencies). We develop our ``spectral complexity'' metric by using a Wasserstein-type distance on
spectral distribution of the recurrence matrix of the directed graph (for an application of such an approach to
measure uncertainty, see \cite{MezicandRunolfsson:2008}).

Based on the above spectral complexity approach, we then develop a novel graph decomposition technique that is based on cyclic interaction between subsystems and does not resort to  symmetrization of the underlying matrices. This approach facilitates the identification of strongly interacting subsystems that can be used for design and analysis of complex systems. In particular, our goal is to group subsystems that should be co-designed or co-analyzed. Our methodology can be viewed as a complementary approach to Fiedler based methods and can also be used to provide graph sparsification \cite{Cohenetal:2017}.

The problem of structural decomposition, clustering or partitioning graphs (or data) into disjoint groups, is a problem that arises in numerous and diverse applications such as, social anthropology \cite{Kottak}, gene networks \cite{GeneSpectral}, protein sequences \cite{ProteinSpectral}, sensor networks \cite{Ali}, computer graphics \cite{CompGraph} and Internet routing algorithms~\cite{KempeMcSherry08}. In general, the problem of clustering requires one to group a set of objects such that each partition contains similar objects, or objects that are ``close'' to one another with respect to an appropriate metric. Alternatively, graph partitioning can be mathematically posed as the set of approaches that minimize the number of edges that cross from one subgroup of nodes to another while maintaining a balanced decomposition~\cite{Tutorial}.

Graph clustering is a well studied topic and spectral clustering has emerged as a very popular approach for decomposing graphs~\cite{Tutorial}. These methods
for clustering graphs use the eigenvalues and eigenvectors of the graph Laplacian matrix to assign nodes to
clusters~\cite{Tutorial}. The theoretical justification for these methods was given by M. Fiedler (see~\cite{fiedler73, fiedler75}). In spectral graph partitioning, one computes the eigenvector corresponding to
the smallest non-zero eigenvalue of the Laplacian matrix. This eigenvector is known as the Fiedler vector
\cite{Tutorial} and is related to the minimum cut in undirected graphs \cite{fiedler73, fiedler75}. The
signs of the components of the Fiedler vector can be used to determine the cluster assignment for the nodes in the
graph \cite{Tutorial}.

The drawback of spectral clustering and other traditional partitioning methods is that they are restricted to
undirected graphs \cite{Tutorial} (they assume that the adjacency matrix is symmetric). The problem of clustering
undirected graphs has been well studied (we refer the reader to \cite{chung_book, wave-clustering1, biggs, everitt,jain,wave-clustering2,schaeffer}). However, for many applications, the adjacency matrix
resulting from the underlying graph representation is not symmetric. Examples include, engineering systems \cite{zeidner2},
social networks, citation networks, Internet communications, and the World Wide Web to name a few~\cite{web}.

The theory for spectral partitioning of directed graphs has not been developed as extensively as that for
undirected graphs~\cite{brualdi:2010}. In~\cite{chung}, the graph Laplacian for directed graphs is defined and its
properties are analyzed. The Cheeger inequality for directed graphs is also derived in~\cite{chung}. In~\cite{gleich} the author extends the work in \cite{chung} to partition directed graphs. A method for clustering
directed weighted graphs, based on correlations between the elements of the eigenvectors is given in
\cite{capoccia}. In \cite{mp}, spectral clustering for directed graphs is formulated as an optimization problem.
Here the objective function for minimization is the weighted cut of the directed graph. In \cite{leicht},
communities or modules in directed networks are found by maximizing the modularity function over all possible
divisions of a network. The heuristic for this maximization is based on the eigenvectors of the corresponding
modularity matrix. Recently, in~\cite{yin2017local}, the authors develop a fast local approach to decompose graphs using network motifs. There are recent papers that 
consider complex eigenvalues of the graph transition matrix to achieve clustering \cite{vanlierdeetal:2018,klymkoandsanders:2016}. While \cite{klymkoandsanders:2016} concentrates on 3-cycles in a directed graph, our methods enable detection of {\it more general, almost-cyclic structures}. The spectral decomposition that we develop in this paper looks beyond the Fiedler vector for partitioning. We utilize complex eigenvalues of the graph transition matrix to identify underlying cycling behavior. The methods of \cite{vanlierdeetal:2018} are closer to ours.\footnote{The paper \cite{vanlierdeetal:2018} appeared in print and on arXiv after our submission. Also, the clustering methodology we provide was first disclosed in an internal report to DARPA \cite{DARPA_project}. We provide a different algorithm for clustering, and provide a more general theoretical justification for the method based on the work in \cite{Mezic:2004}. }

The paper is organized as follows. In section~\ref{comp} we introduce the idea of spectral complexity of a directed graph. We compare the new measure of complexity to the
standard graph energy complexity metric used in literature. In section~\ref{clust} we propose an approach for partitioning
directed graphs that groups nodes into clusters that tend to map into one another (i.e. form ``almost cycles''). In section 4
we give examples and compare our results with existing methods.

\section{Spectral Complexity}
\label{comp}

The key idea underlying our methodology is that every digraph $G=(V,A,B)$, where 
$V$ is a set of nodes, $A$ is a set of directed edges, and $B$ is a set of weights, can be represented  using a multi-valued (one-to-many) map $T:V\rightarrow V$ that maps  node $i$ to a set of nodes $j\in V_i$, with the associated probabilities $p_{ij}=\beta_{ij}/\sum_j \beta_{ij}$, $\beta_{ij}$ being weights.  If a node is a sink, and has no edges, we set $p_{ii}=1$. We consider the weighted adjacency matrix $U$ whose $i,j$ element is $p_{ij}$. This matrix is analogous to the Koopman operator in dynamical systems \cite{Koopman:1931,MezicandBanaszuk:2004}. 

We can decompose  the state-space of one-to-many maps into the recurrent set $V_r$ and non-recurrent set $V_{nr}$. We define the recurrent set as the set of all the points $k \in V$ such that {\it every} orbit that starts at $k$  lands in $k$ some time later. The rest of $V$ is the transient (non-recurrent) set. 
Obviously, the row-stochastic matrix $U$ has its restriction $R$ to the recurrent set, where $R$ is obtained from $U$ by deleting the rows and columns corresponding to transient set nodes. Note that $R$ is also row-stochastic, since nodes in the recurrent set have $0$ probability of transitioning to the transient set. We call $T$ irreducible if we can get from any initial state $k$ to any final state $l$, i.e. $l \in T^n(k)$ for some $n$ and every $k,l$. $V$ can be split into irreducible components. It can be shown that on each irreducible component, every state has the same {\it period} where in case the period is the greatest common divisor of all $n$ such that $k\in T^n(k)$ \cite{grimmettandstirzaker:2001}. We identify complex vectors  with elements $v_j,j\in V$ with functions $f:V\rightarrow \C$ such that $f(j)=v_j,j\in V$. The level set of $f$ is a set $C_c$ in $V$ such $f=c, c\in\C$ on $C_c$, i.e. the function has a constant value on $C$.  In the following, we will use the notion of period $p=k/j$, where $k,j$ are integer and $k\geq j$ to mean $p=k$ if $k/j$ is not an integer, and $k/j$ otherwise. We have the following theorem:
\begin{theorem} Let $T$ be irreducible of period $d$. Then:
\begin{enumerate}
\item $\lambda_1=1$ is an eigenvalue of $U$ and $R$. The eigenspace of $\lambda_1$ is one-dimensional and consists of constant functions.
\item $\lambda_{jd}=e^{i2\pi j/d}$ is an eigenvalue of $U$ and $R$, where $j=1,...,d$. The eigenspaces associated with each of these consist of vectors whose level sets define an invariant partition of period that is equal to $d/j$.
\item The remaining eigenvalues of $U$ satisfy $|\lambda_j|<1$.
\item If there is a pure source node, then $0$ is in the spectrum of $U$.
\end{enumerate}
\label{the:1}
\end{theorem}
\begin{proof}
Items $1.$ and $3.$ are a simple consequence of the Perron-Frobenius theorem \cite{katok}. Item 2. follows from the observation \cite{grimmettandstirzaker:2001} that a Markov chain with period $d$ possesses eigenvalues $\lambda_{jd}=e^{i2\pi j/d}$, and from the fact that $T$ is a Discrete Random Dynamical System \cite{MezicandBanaszuk:2004}. Then,  Theorem 15 implies\footnote{Note that, following the proof in Appendix 1 of \cite{MezicandBanaszuk:2004} the reversibility condition can be relaxed.} that the associated eigenfunction $f_{j/d}$ is a deterministic factor map of $T$. The theorem also implies that the state space $V$ splits into sets on which $f_{j/d}$ has constant value. The number of such sets is $d$ provided $d/j$ is not an integer, and $d/j$ if it is. The last statement follows from the fact that if $i$ is a source node, then a vector that is $1$ on $i$ and $0$ on all other nodes gets mapped to $0$ by $U$.
\end{proof}

If $T$ is not irreducible, it can always be split into irreducible components, and then Theorem \ref{the:1}
 can be applied on each component.
In Theorem \ref{the:1} the cycle of order $d$ is identified and its eigenvectors serve to partition the graph by using their level sets. The lower order cycles are also associated with an eigenvalue and an associated partition:
\begin{theorem}  If $\lambda_{jd}=e^{i2\pi j/n}$ is an eigenvalue of U or $R$, where $n\leq d$ then the eigenspace associated with it consist of vectors whose level sets define an invariant partition of period that is equal to $n/j$.
\label{the:2}
\end{theorem}
\begin{proof}
This also follows from Theorem 15 in \cite{MezicandBanaszuk:2004} if we take the state space to be a discrete space $V$ of $n$ nodes, and $T$ as a random dynamical system on it. 
\end{proof}

Theorems \ref{the:1} and \ref{the:2} give us motivation to define a measure of complexity based on the structure of recurrent (i.e. cycle-containing) and non-recurrent sets. Here are the postulates that we use for defining  complexity, that is based on the properties of $T$:
\begin{enumerate}
\item Any graph that consists of disconnected single nodes has complexity equal to the sum of complexities of the nodes.
\item Any linear chain has complexity equal to the sum of the sum of complexity of the nodes and weights of the edges.
\item Complexity of a graph that has no non-recurrent part and $n$ nodes is measured as a distance of distribution of eigenvalues of $U$ to delta distribution at $1$ - called the spectral complexity - added to the sum of the complexity of the nodes.
\end{enumerate}
Note that in the definition of spectral complexity we use the notion of distance on the unit disk. There are a variety of choices that can be made, just like the choice of $L^1$ norm or $L^2$ norm on Lebesgue spaces.
We now describe definitions and algorithms for computation of complexity, with a specific choice of distance based on the Wasserstein metric.

\subsection{Definition of Spectral and Total Complexity of a Directed Graph}

In this section, we propose an algorithm for calculating the complexity of directed graphs using the spectral
properties of the matrix $R$. To construct the matrix for a graph, we start by removing all the sources and their corresponding edges until no sources are left. This is motivated by the notion that sources are elements that
contribute to complexity in a linear manner, and will be included in the complexity metric through the edge weights.
We note that, a source is a node with only outgoing edges (a disconnected node is not a source). In matrix terms,
every source contributes to a zero (generalized) eigenvalue. We then construct the edge weighted adjacency matrix
for the new graph that effectively captures the dynamics of the multivalued map $T$ (a random walk on the graph). Thus, the rows of this adjacency matrix are normalized, such that the sum of the elements in any
given row is $1$. This is achieved by dividing each element in the row by the sum of all the row elements. If the
row contains only zeros (the given node is a sink), we put a $1$ on the diagonal element in that row, i.e. we add a
self loop in a standard manner of associating a Markov chain with a graph. This changes the zero eigenvalue associated with that row to $1$. Note that the eigenvector associated
with this eigenvalue is constant on the connected component, and all the other eigenvalues and
eigenvectors remain unchanged. We call the resulting matrix R the {\it recurrence matrix}. As a corollary of Theorem \ref{the:1} we have that this matrix always has an eigenvalue $1$ associated with a
constant vector, and all of the remaining eigenvalues are distributed on the unit disk.

We now define a complexity measure on the class of recurrence matrices. For a $K\times K$ recurrence matrix, we
will define the least complex matrix to be the identity matrix (this matrix corresponds to a graph with no edges). This corresponds to the graph in which each node
only has a pure self-loop. We define complexity as the distance of the eigenvalue distribution of $R$ from the eigenvalue
distribution of the identity matrix.
Distance on distributions can be measured in different ways. Here we adopt an approach based on the Wasserstein
distance. For this, we first need to define a distance on the unit disk. We do this using polar coordinates $r$
and $\theta$,
considering the unit disk as the product space $I\times S^1$, where $I=[0,1]$. The distance on $I$ is the usual
one $d(r_1,r_2)=|r_1-r_2|$, while on $S^1$ we impose the discrete metric:
$$
d(\theta_1,\theta_2)=\left\{
\begin{array}{cc} 1 & \mbox{if} \ \theta_1\neq \theta_2, \\
                            0 & \mbox{if}  \ \theta_1= \theta_2. \\
\end{array}
\right.$$
Now, the normalized Wasserstein distance between the least complex eigenvalue distribution and the one with
eigenvalues $\lambda_i=(r_i,\theta_i), \ i=1,...,n$
is,
\begin{equation}
F=\frac{1}{K} \left(\sum_{i=1}^K (1-r_i)+\mathbbm{1}_{\{\theta\neq 0\}}(\theta_i)\right),
\label{eq2}
\end{equation}
where $K$ is the number of non-zero eigenvalues of the recurrence matrix $R$ and $\mathbbm{1}_{\{\theta\neq 0\}}$
is the indicator function on the set $\{\theta\neq 0\}$.  The following fact on the graph with least spectral complexity is obvious:
\vskip .3cm
{\bf Fact:} The graph with $K$ disconnected nodes has complexity $0$.
\vskip .3cm
The first term in the spectral complexity function~(\ref{eq2}) is a measure of the amount of ``leakage'' in the
graph. If one is performing a random walk on the graph then the leakage is a measure of the probability of
transition between nodes \cite{Metastab}. This term takes values between 0 (no leakage) and 1 (probability of
transition is $1$). In other words, the first term captures the decay in probability of a random walk and the second term captures the cycles. According to the above definition,  the maximally complex graph in some class should maximize both terms separately. Namely, the eigenvalues of such a graph would be radially as close to zero as the class definition allows, and would have the maximal number of eigenvalues off the positive real line inside the unit disc, thus maximizing the second term. 
The following result indicates how the maximum spectral complexity of a graph is achieved if the graph family is not restricted:
\begin{theorem}  Let $R$ be a $K\times K$ recurrence matrix of a $K$-node graph. Then maximal spectral complexity $F$ is achieved for a matrix with constant entries.
\label{the:3}
\end{theorem}
\begin{proof} A recurrence matrix $R$ with constant entries has  $K-1$ zero eigenvalues corresponding to eigenvectors that have $-1$ at $j$-th component  and $1/(K-1)$ for all other components. Note that these  are counted as $\theta\neq 0$ eigenvalues. Since $1$ is always an eigenvalue,  the resulting eigenvalues maximize both the first and the second sum in $F$, making it $2(K-1)/K$.
\end{proof}

From the above theorem, it is clear that  graphs with a large number of nodes have maximal complexity very close to $2$. This is evident in the example we present in the next section.

If the entries of $R$ are such that it forms a random Markov matrix \cite{Bordenaveetal:2012}, then, as we prove next, the complexity increases to maximal complexity as the size of the matrix increases. 
\begin{theorem} \label{the:rand} Let $\beta_{jk},j,l\geq 1$ be i.i.d random variables with bounded density, mean $m$ and finite positive variance $\sigma^2$. Every realization of $\beta_{jk},j,l\leq K$ gives a weighted directed graph. Let $R$ be a $K\times K$ recurrence matrix of such a $K$-node graph. Then  $F(R)\rightarrow 2$ as $K\rightarrow \infty$, with probability $1$. \label{the:4}
\end{theorem}
\begin{proof} The  recurrence matrix $R$ is a random Markov transition matrix \cite{Bordenaveetal:2012}
with the underlying Markov chain irreducible with probability $1$. Let 
\be
\mu_{\sqrt{K}R}=\frac{1}{n}\sum_{j=1}^{n}\delta_{\lambda_j},
\ee be the empirical measure supported on the location of eigenvalues of the matrix $\mu_{\sqrt{K}R}$,
where $\delta_{\lambda_j}$ is the Dirac delta function centered at eigenvalue $\lambda_j$.
Then, Theorem 1.3 in \cite{Bordenaveetal:2012} implies 
that $\mu_{\sqrt{K}R}$ converges to the uniform measure on the disk ${\cal U}_{\sigma/m}=\{z\in \C | z\leq \sigma/m\}$.
This, in turn, implies that the modulus of eigenvalues of $R$ goes to zero as $K\rar \infty$, and that 
\be
\displaystyle \lim_{K\rar \infty}\frac{1}{K} \left(\sum_{i=1}^K (\mathbbm{1}_{\{\theta\neq 0\}}(\theta_i)\right)=1
\ee
Also noting that $\displaystyle \lim_{K\rar \infty}(K-1)/K=1$, we conclude the proof.
\end{proof}

The above result is interesting in the context of numerical tests that we do in section \ref{sec:compgraphenergy}, which show random graphs of increasing size whose complexity converges to $2$, and in the section \ref{largenet}, where most of the eigenvalue distributions for several web-based networks are within a disk in the complex plane, but a small proportion is not, indicating the non-random nature (and lower complexity) of these networks.

The use of the ``counting" of eigenvaues with $\theta_j\neq 0$ in the second term of $F$ makes the spectral complexity measure have some features of discrete metrics, as the following example shows:
\begin{example}[Spectral complexity in a class of recurrent 2-graphs] 
\label{2-graph}
We consider graphs with 2 elements that have both a self loop and an edge connecting them to the other element, with uniform probabilities as shown in figure \ref{fig:2Graph}.

\begin{figure}[!h]
\begin{center}
\includegraphics[clip=true, trim=100 150 100 150, scale=0.3]{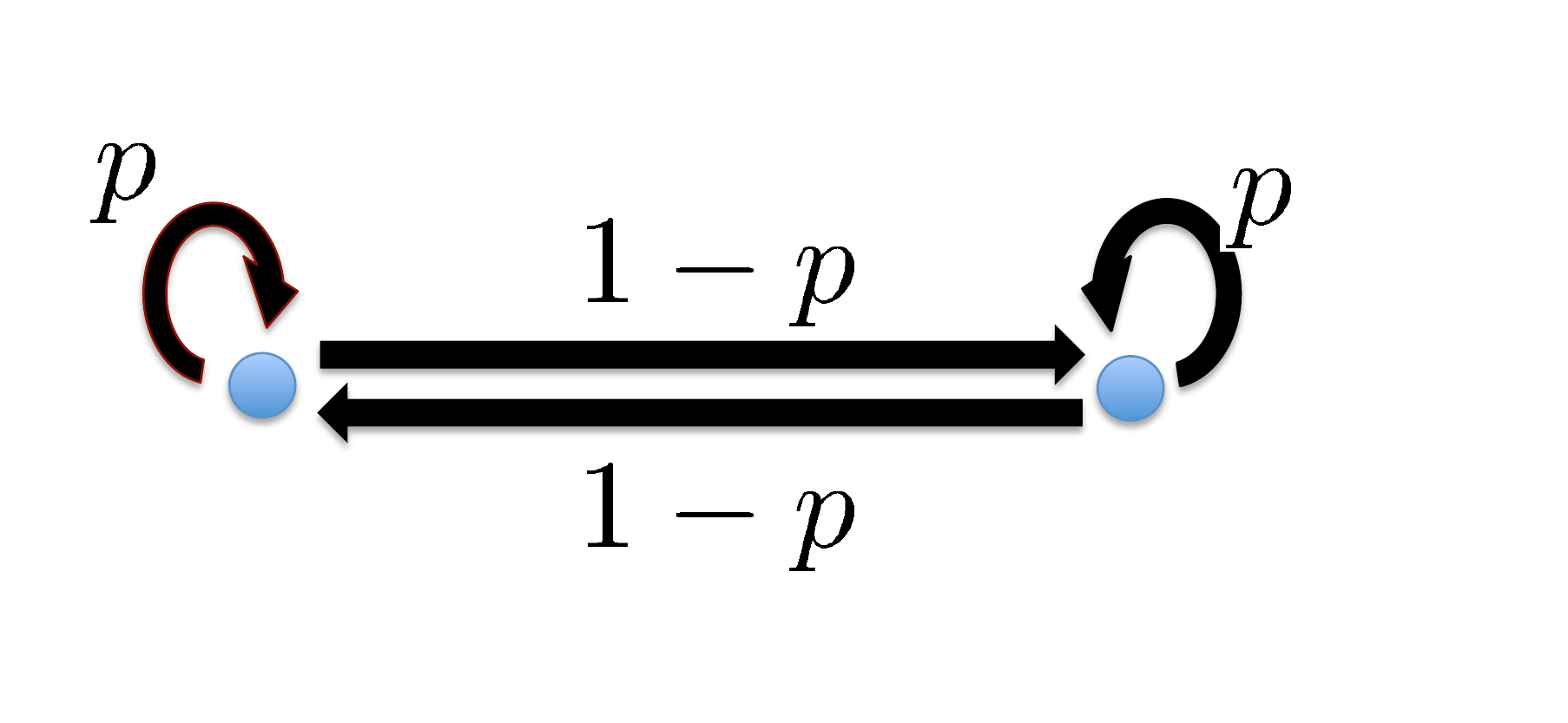}
\end{center}
\caption{\label{fig:2Graph} Graphical representation of the family of graphs with two nodes, equal strength self loops, and equal strength connecting edges.}
\end{figure}

Such a system has $R$ of the form:
\begin{equation}
R=\begin{bmatrix}
p & 1-p \\
1-p & p
\end{bmatrix}
, \end{equation}
where $p\in[0,1]$. The eigenvalues $\lambda$ of $R$ satisfy  the equation 
\be
(p-\lambda)^2=(1-p)^2.
\ee
One solution comes from,
\be
p-\lambda=1-p \ \Rightarrow \lambda=2p-1,
\ee
and the other from,
\be
p-\lambda=p-1 \ \Rightarrow \lambda=1.
\ee
For $p<1/2$ the self loop is weaker than the edge connecting to the other node, and for $p>1/2$ the opposite is true. The spectral complexity is
\be
F=\left\{
\begin{array}{cc} (1-|2p-1|+1)/2=(1+2p)/2, & \mbox{if} \ \ p\leq 1/2 \\
                            (1-(2p-1))/2 = 1-p & \mbox{if}  \ \ p> 1/2 \\
\end{array}\right.
\ee
Spectral complexity of this class of graphs distinguishes between graphs that have stronger self-interaction than interaction between the nodes, characterized by $p>1/2$ and the graphs in which the interaction between the nodes is stronger than the self-interaction.
Note that spectral complexity is discontinuous at $p=1/2$. This is in line with the behavior of the underlying Markov chain: for $p>1/2$ any initial probability distribution on the chain will decay exponentially and monotonically to the uniform distribution. For $p<1/2$, the decay of the distribution assumes oscillatory manner, thus representing a qualitative, discontinuous change in behavior. Note that for $p=1/2$ the complexity measure shows features of a discrete metric. Thus, the discontinuity in the complexity metric accurately captures the transition from the more complex oscillatory evolution of the distribution to the invariant measure (for $p\leq1/2$) versus the less complex monotonic convergence to the invariant measure for $p>1/2$. We note that the oscillatory nature of the distribution, in the more complex case, corresponds to strong interaction between nodes (since $p\leq 1/2$). This is in contrast with the weak interactions between nodes in the $p>1/2$ case, whereby the graph interactions are less important when compared to the self interaction of nodes.
\end{example}

\subsection{Physical intuition for complexity metric and meaning of eigenfunctions of the recurrence matrix for the network behavior}
Spectral objects associated with undirected graphs -such as the Fiedler eigenvalue, that is associated with speed of mixing of the associated Markov chain and reflects connectivity of the underlying graph, and the Fiedler vector, whose components indicate subgraphs that have strong internal connectivity but weak interconnectivity -  often have impact on the physical understanding of the network. The same is true for the eigenvalues and eigenvectors of the matrix $R$. They have strong correlation with the structural properties of the underlying graph. For example, existence of a real eigenvalue $0>\lambda\geq -1$ indicates that the network can be split into two subnetworks that have weak internal connectivity but strong interconnectivity between two subnetworks (see Example \ref{2-graph}). Also, the associated eigenvector values can be clustered into two separate sets that indicate the mentioned subgraphs. Both the simple example \ref{2-graph}, and the large graph Wikipedia example in the section \ref{largenet} provide evidence for this statement. Analogously, an eigenvalue set $\lambda_1,\lambda_2.\lambda_3$, whose arguments are close to $(0,\pi/3,2\pi/3)$ indicates that the graph possesses 3 subgraphs with weak internal and strong connectivity between the 3 subgraphs. An example of this is shown in the section \ref{largenet} for the Gnutella network. 

The complexity metric has the above spectral elements as part of the metric. It speaks to the structural complexity of the graph, but it has a  physical meaning for the behavior of the network as well. As a simple example, consider the case of spring mass system illustrated in figure \ref{fig:springmass}. If we set $k_1=k_2$, the weight matrix is
\begin{equation}
R=\begin{bmatrix}
k_1 & k_{12} \\
k_{12} & k_1
\end{bmatrix}
\end{equation}
The associated recurrence matrix is then
\begin{equation}
R=\begin{bmatrix}
p & 1-p \\
1-p & p
\end{bmatrix}
\end{equation}
where $p=k_1/(k_1+k_{12}).$

\begin{figure}[!h]
\begin{center}
\includegraphics[clip=true, trim=0 0 0 0 , scale=0.3]{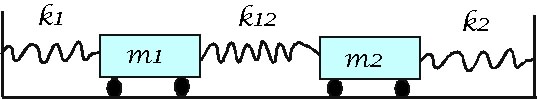}
\end{center}
\caption{\label{fig:springmass} Graphical representation of the mass-spring system.}
\end{figure}
Now assume $p=1$. This indicates a decoupled system, and the complexity of such system is clearly the smallest among all considered systems. For $p$ slightly smaller then $1$, the complexity is small, as the system is ``almost decoupled".  For $p=0$, the system has one eigenvalue at $-1$, indicating that the $2$ masses interact strongly, while there is no self-interaction for either mass.  It is physically intuitive that the highest complexity occurs for $p=1/2$, in which case the effects of both the spring attached to only one of the masses and the spring attached to both masses have equal influence on the individual mass motion. It is also intuitive that the situation with $p=0$ is less complex - for example, in design considerations we do not need to take into account   the properties of two of the springs. 
This intuition carries over to other examples. If we take three masses with no self interaction, but connected by springs, there is a double eigenvalue at $-1$ and thus its complexity is larger than that of the 2-mass system. The more balanced the self connectivity is with the connectivity to other nodes the more complex tasks like engineering design will become. It is sometimes argued that networks with full connectivity are simpler to analyze, but this comes from a statistical mechanics approach to the problem. In a design engineer or maintenance engineer world, adding an edge in the device or network design {\it always} increases the complexity of the resulting system.

The above discussion introduces a way of measuring the complexity of the recurrent part of a directed graph, and points to the intuitive aspects of the definition. But, complexity of the graph is not solely a function of the recurrence and cycles. Namely, more components in a graph, and more edges between non-recurrent nodes contribute to complexity as well - and we assume they do so in a linear fashion. 
Thus, if for a particular application we need to take into account the weights of nodes and the weights of the removed edges while removing sources,
the total complexity $C$ can be formulated in the following way:
\begin{equation}
C=\frac{\gamma\Big(\displaystyle\sum_{i=1}^N \alpha_i + \sum_{j=1}^M \beta_j\Big) + WF}{1+W},
\label{eq3}
\end{equation}
where $W$ is the user-defined weighting parameter for the spectral complexity in the total complexity metric that can take any
value from $\left[0,\infty\right)$. $N$ is the initial total number of nodes, $\alpha_i$'s are the complexity of
the individual nodes,\footnote{This is obtained either as user input or by some measure of complexity of dynamics on the individual node - e.g. through the use of the spectral distribution associated with the Koopman operator of the dynamical system \cite{MezicandBanaszuk:2004}.} $M$ is the number of edges removed while removing source nodes and $\beta_j$'s are the
weights of the edges that were excluded in the source nodes removal step. $F$ is given by equation (\ref{eq2}) and
$\gamma$ is the scaling factor that arises due to the fact that the terms $\Big(\displaystyle\sum_{i=1}^N \alpha_i + \sum_{j=1}^M \beta_j\Big)$ and $F$ might have vastly different numerical values. One choice for $\gamma$ is the following:
\begin{equation}
\gamma=\frac{\mathbb{E}(F)}{\mathbb{E}\displaystyle \Big(\sum_{i=1}^N \alpha_i + \sum_{j=1}^M \beta_j\Big)}.
\label{eq:sc}
\end{equation}
Note that the expectation is taken over various configurations of the system, and thus the 
probability distribution on a collection of graphs must be given. An alternative choice is to replace the $\mathbb{E}$ operator with the nonlinear {\it max} operator in equation (\ref{eq:sc}).

\begin{example} [Unicyclic directed graphs] Consider the family of unicyclic connected graphs, with nodes $v_j,j=1,...,N$ and edges $e_{kj}=0, k<N-1, j\neq k+1$, $e_{j,j+1}=1,j=1,...,N-1$ and $e_{N,1}=1, e_{N,j}=0, j\neq 1$ (see figure \ref{fig:unicycle}). Let the complexity of individual nodes be $1$, and $\gamma=1$. Then the complexity is equal to $C=(2N+W)/(1+W)$ and increases monotonically with the size of the graph.

\begin{figure}[!h]
\begin{center}
\includegraphics[clip=true, trim=140 140 140 140, scale=0.26]{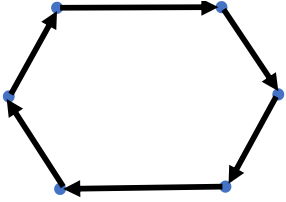}
\end{center}
\vspace{-1cm}
\caption{\label{fig:unicycle} Graphical representation of the family of unicycle directed graphs.}
\end{figure}
\end{example}

\subsection{Comparison with Graph Energy}
\label{sec:compgraphenergy}
In this section, we compare the spectral complexity introduced in this paper to graph energy. The notion of graph energy
\cite{graphenergy1, graphenergy2} emerged from molecular and quantum chemistry, where it has found use in
ranking proteins on the basis of the level of folding \cite{graphenergy-chem}. It has also been used as a metric
for complexity of graphs. The graph energy complexity, interestingly, does not peak for graphs with maximum
possible connections (the rank of the adjacency matrix for a complete graph is not maximum). Instead,
statistically the most complex graphs are those with $\approx 2/3$ possible connections \cite{graphenergyresults}.
Note that this complexity metric fails to capture directed cycles in the graph since one is forced to either work
only with undirected or symmetrized directed graphs, as demonstrated below.

The algorithm for calculating graph energy is as follows. At first, for a given graph, we construct the adjacency
matrix $M$:
\begin{equation}
M_{ij}=\left\{
         \begin{array}{ll}
           1, & \hbox{ for all edges $(i,j), i \neq j$ of the graph;} \\
           0, & \hbox{ otherwise.}
         \end{array}
       \right.
\label{eq5}
\end{equation}
The graph energy $C$ is calculated by using the following formula:
\begin{equation}
C=\Big(\frac{1}{|A|} \sum_{k=1}^{|A|} b_k \Big) \sum \mbox{SVD} (M), 
\label{eq6}
\end{equation}
where $b_{|A|}$ are edge weights, $|A|$ is the number of edges in the graph, $\mbox{SVD}(M)$ is a vector of singular
values of matrix $M$. Equation (\ref{eq6}) can  be used with symmetrized adjacency matrix $M^{(sym)}_{ij}=M_{ij}{\vee}M_{ji}$, where ${\vee}$
is the logical OR operator.

In the following, we present Figures \ref{complexity} and \ref{energy} to highlight the difference between the
complexity introduced in this paper and the graph energy. Random graphs were probabilistically constructed using
the following formula: the probability with which a node is connected to another node is given by
\begin{equation}
p=\frac{\hbox{Average Degree }}{\hbox{Number of Nodes}}.\nonumber
\end{equation}
All graphs considered have 1000 nodes. The average degree was varied from $1$ to $20$ in increments
of $1$ and then from $50$ to $1000$ in increments of $50$. The degree is defined as the number of outgoing edges
from each node. All weights of the edges are equal to 1. Each realization was repeated 10 times.

\begin{figure}[!h]
\begin{center}
\includegraphics[scale=0.6]{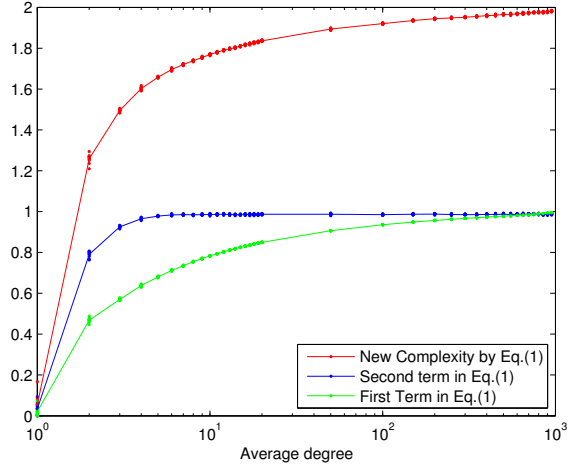}
\end{center}

\caption{\label{complexity} Complexity computed using equation (\ref{eq2}) (red), second term in equation (\ref{eq2}) (blue)
and first term in equation (\ref{eq2}) (green) as a function of the average degree of the graph. Each graph has 1000
nodes.}
\end{figure}

The spectral complexity increases fast with the average degree, reaching values of about 1.8 (out of the maximum
possible value of 2) at an average degree of about 20/1000 of the total number of nodes; it then continues to
increase monotonically, but less rapidly, with the average degree.

\begin{figure}[!h]
\begin{center}
\includegraphics[scale=0.6]{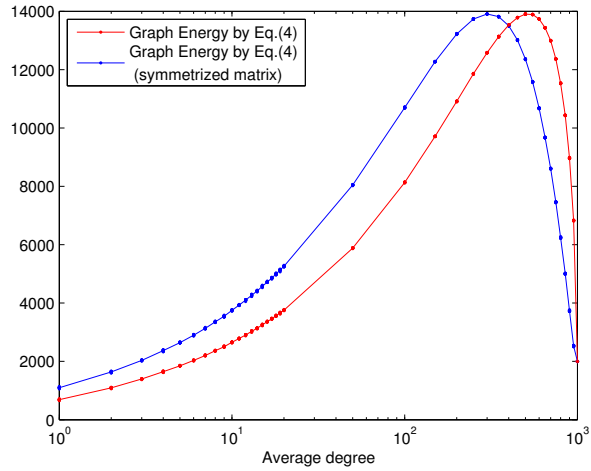}
\end{center}

\caption{\label{energy} Graph energy calculated by equation (\ref{eq6}) as a function of the average degree of the
graph. In the case when the matrix was symmetrized, the average degree relates to the initial non-symmetrized
matrix. Each graph has 1000 nodes.}
\end{figure}

In the case of the graph energy, as shown in Figure \ref{energy}, the maximum energy is reached when the average
degree is at about 50\% of the total number of nodes; then the graph energy starts to decrease.

This difference can be understood from the following argument. For simplicity, we take graphs 
with edge weights all equal to $1$. As shown in Theorem \ref{the:rand}, random graphs with  large average degree will statistically have eigenvalues
with modulus close to zero. Since graph energy is equal in this case to the sum of moduli of eigenvalues, the graph energy will be small.
In other words, small graph energy is in fact {\it indicative} of a large number of connections in the graph, and thus large, not small, complexity.
Namely, the key to decrease of energy of random graphs is the decrease in the moduli of the eigenvalues. In contrast, the metric F
counts the number of complex eigenvalues, that will, in the case of a random graph with large average degree tend to increase with the average degree.

In the case of graphs corresponding to engineered systems, there is no reason why the complexity should decrease
with increasing the number of connections (interdependencies) in the graph. Thus, we believe that the complexity
measure introduced in this paper is more appropriate for \textit{engineering and physical systems}. We note that the graph energy metric is more appropriate from an information theory standpoint.

We additionally note that in~\cite{berwanger}, the authors develop a complexity measure that is based on the entanglement of cycles in directed graphs. They compute this metric using a game theoretic approach (using a cops-and-robbers game). This reachability approach is similar in spirit to our spectral cyclomatic complexity measure. However, we note that, in general, computing $k$-entanglements scales as $O(n^{k+1})$, whereas our approach in general scales as $O(n^3)$, and much faster than that for sparse graphs. The spectral complexity captures the ``entanglements'' at all scales of the graph (for all $k$).  Moreover, unlike the approach in~\cite{berwanger}, our methodology leads to natural clustering of the graph that is discussed in the next section.

\section{Clustering of Directed Graphs}

\label{clust}

The clustering of undirected graphs is a well-developed area~\cite{chung_book, Tutorial,wave-clustering1} with several decades of research behind it. The area of clustering of directed graphs is far less developed. However, the analysis and clustering of directed graphs is slowly coming in vogue~\cite{communitydirected,chung,blockmodel1,blockmodel2}. In~\cite{chung}, the author generalizes random walk based Cheeger bounds to directed graphs. These bounds are related to the spectral cuts often used for graph partitioning~\cite{chung_book}. In~\cite{communitydirected}, the authors generalize Laplacian dynamics to directed graphs, resulting in a modularity (quality) cost function for optimal splitting. 

An alternative approach has focused on block modeling~\cite{blockmodel1,blockmodel2}. Under this methodology, nodes are grouped into classes that exist in an image graph. This assignment is performed based on node connectivity and neighbor properties. This approach assumes that a template image graph and roles (for the nodes) are supplied a priori. The graph is then fit onto the image graph using an optimization scheme~\cite{blockmodel1}. Although the approach extends to directed graphs, such image graphs are not always available in engineering or social systems. 

In the following, we introduce a new graph clustering approach that complements standard spectral methods for decomposing graphs. In particular, we construct a new algorithm that is based on computing the underlying cycles in the graph by computing the corresponding generating eigenvalues and eigenvectors. In particular, by decomposing the graph into these cycles, we aim to identify strongly interacting components in a directed graph. The method is compared to Cheeger and Laplacian dynamic based methods~\cite{chung,communitydirected}.


From the discussion leading to Theorem \ref{the:1}, we recognize that  cycling in a directed graph is associated with its recurrent part. Thus, we can use spectral properties -- and in particular complex eigenvalue pairs -- of the recurrence matrix $R$ in order to recognize cycles in a directed graph. Note that, according to Theorem \ref{the:1}, a set of complex eigenvalues with unit modulus always has a generator $e^{i2\pi/d}$. We extend this idea to eigenvalues off the unit circle and search for such generating eigenvalues.

In our algorithm, we seek the dominant cycle in a graph by identifying an eigenvalue (the generating eigenvalue) that is closest to a pure cycle on the unit circle.   The algorithm is as follows:  we compute nonzero eigenvalues $\lambda_j$ of $R$. We then compute the angles
$\alpha_j$ of the calculated eigenvalues in the complex plane and set
\begin{equation}
K_{min}=\displaystyle \mbox{argmin}_K \Big\{\frac{1}{K-1} \sum_{t=2}^K \mbox{min}_{j \in S} \Big|\exp(\frac{2\pi
i}{K} (t-1)) - \lambda_j)\Big|\Big\},
\label{eq4}
\end{equation}
where $K=2,\dots,N$, $N$ is the number of nonzero eigenvalues, $S$ is the set of eigenvalues for which
$(2\pi/K)\times (t-1.5) \le \alpha_j \le (2\pi/K)\times(t-0.5)$. If the set $S$ is empty, then the minimum in equation
(\ref{eq4}) is 1. We denote the number of clusters corresponding to the dominant cycle as $K_{min}$. Then we find the generating eigenvalue(s) and the
corresponding eigenvector(s). We choose $j$ such that $\pi/K_{min} \le \alpha_j \le 3\pi/K_{min}$. We want the
generating eigenvalue to be close to the case of a pure cycle of size $K_{min}$, when the generating eigenvalue is
at $2\pi/K_{min}$. We find the index of the first generating eigenvector as
$\mbox{argmin}_j|\lambda_j-\exp(\frac{2\pi i}{K_{min}})|$. Other generating eigenvalues are those that are within
a predefined threshold (we use $10^{-4}$ in our work) of the first generating eigenvalue.

For each generating eigenvector $v_j$ we compute angles $\phi_i$ in the range $[0,2\pi]$ for each element $1\leq
i\leq N$. Then we obtain graph clusters by partitioning coordinates of $v_j$ into $K_{min}$ groups by splitting
the unit circle into $K_{min}$ equal parts. Disconnected nodes and sinks are placed in separate clusters.

For example, the $7$ node graph (see Fig.~\ref{bins_val} (left)) with $6$ non-zero eigenvalues of the recurrence matrix (red points in
Fig.~\ref{bins_val} (right)) has $K_{min}=3$ clusters. The sector of the unit circle, that contains the generating
eigenvalue is between $\pi/K_{min}$ and $3\pi/K_{min}$ and is colored with green in Fig.~\ref{bins_val} (right). The generating eigenvalue is the non-zero eigenvalue which is closest to the eigenvalue of the pure cycle of size $K_{min}$.

\begin{figure}[!htb]
	\centering	
	\subfigure{\includegraphics[width=0.58\linewidth]{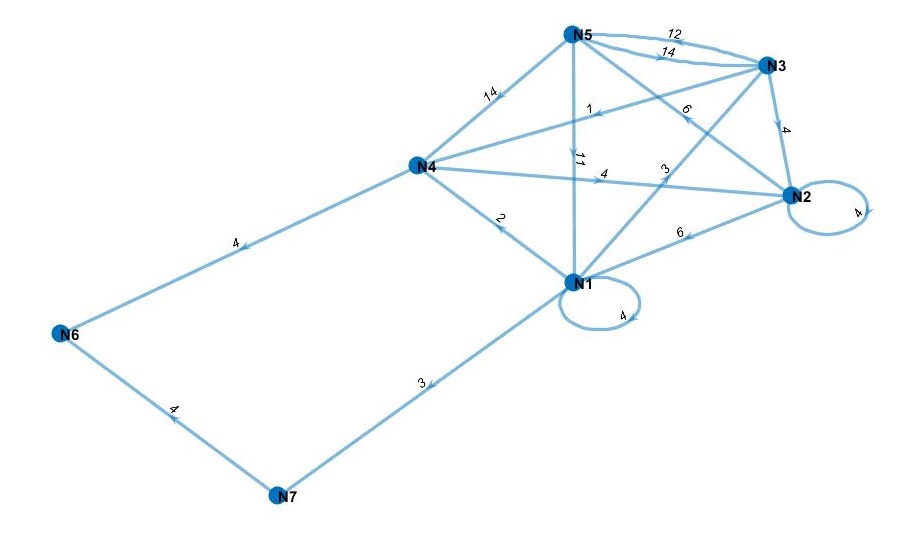}}\quad
	\subfigure{\includegraphics[width=0.39\linewidth]{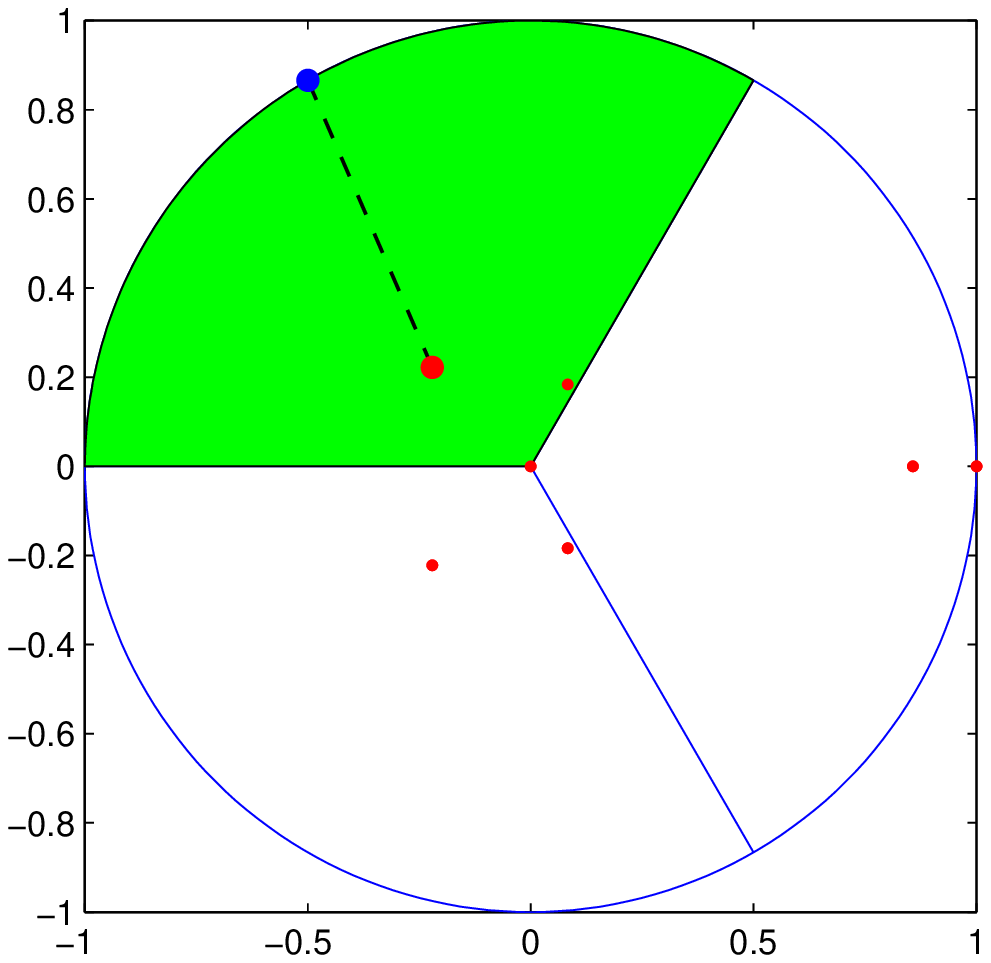}}
	\caption{An example graph (left); Eigenvalues of the recurrence matrix (right). The blue dot is the eigenvalue of the pure cycle of size 3. The big red dot is the generating eigenvalue as it is closest to the blue point within the green sector.}
	\label{bins_val}
\end{figure}

In previous work~\cite{jacobi,jacobigornerup} a method for identifying coarse-grained dynamics using aggregation of
variables or states in linear dynamical systems was developed. The condition for aggregation is expressed as a
permutation symmetry of a set of dual eigenvectors of the matrix that defines the dynamics. It is based on the
fact that the $n \times k$ aggregation matrix $\Pi$ reduces a (transition) matrix P describing a linear dynamical
system if and only if there exists a set of $k$ linearly independent vectors invariant under $P^T$, e.g. (left)
eigenvectors, that respect the same permutation symmetry group as $\Pi$. It is straightforward to identify
permutation symmetries in the invariant vectors of $P^T$. A permutation symmetry is realized through identical
elements in the vectors. Thus, by identifying the above permutation symmetries, one can group elements in a complex (directed) graph. In other words, the algorithm that we introduced above leads to a natural method for graph sparsification \cite{Cohenetal:2017}.

\section{Examples}

\subsection{Fixed Wing Aircraft Example}

To test both our clustering approach and the complexity metric, we consider the architecture of a fixed wing
aeroplane system~\cite{zeidner2}. This is a particularly important and relevant example since
recent analysis by the RAND Corporation concluded that the increase in cost of fixed wing aircraft is primarily
due to increased complexity \cite{randstudy}. An example of the impact of the complexity of fixed wing aircraft is the recent cost overruns of the F-35 platform~\cite{f35}.

The aerospace system considered in this work, consists of the following functional subsystems: aircraft engine,
fuel system, electrical power system (EPS), environmental control system (ECS), auxiliary power unit (APU), ram
cooler, and actuation systems. These subsystems may be connected to one another through various means. For example,
the engine may provide shaft power to the fuel system, the EPS and actuation system. Similarly, the APU may be
connected to the engine as it may be required to provide start up pneumatic power. Note that the interconnections
need not be electrical or mechanical in nature. Since the fuel system can be designed to absorb heat from the
actuation system and EPS, the dependencies of the subsystems may also be thermal. For a discussion on these systems we
refer the reader to~\cite{zeidner2}. An example architecture depicting the subsystems and their
interconnections is shown in Fig.~\ref{Fig:aero}.

\begin{figure}[!h]
\begin{center}
\includegraphics[scale=0.35,bb=0 0 493 528]{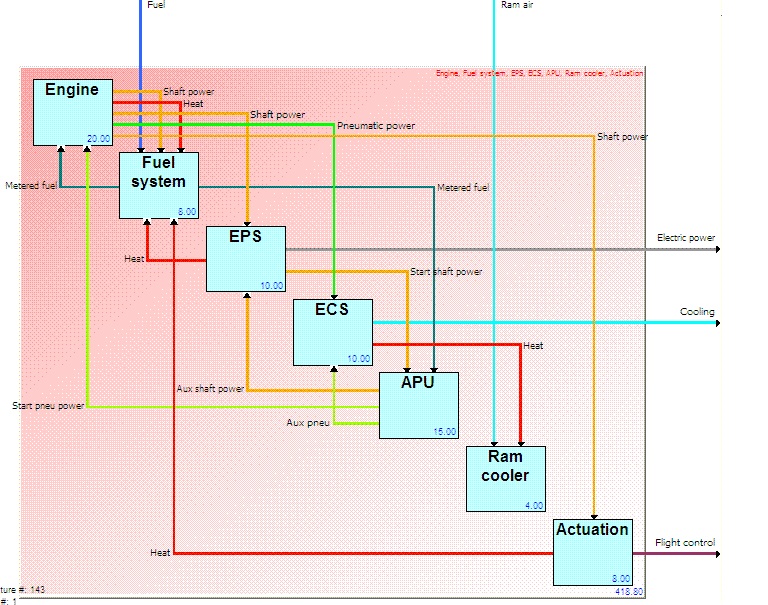}
\caption{Example architecture of a fixed-wing aeroplane system.}\label{Fig:aero}
\end{center}
\end{figure}

Traditionally, aerospace system architectures are specified by subsystems (such as EPS, ECS etc) and their
interconnections. The exploration of design space for these aerospace systems can be a particularly daunting and
challenging task. One possible approach to this problem has been to enumerate all feasible architectures and then
pick the most desirable one \cite{zeidner2}. It would appear that the exponential size of the design space would
make this enumeration task intractable. However, the feasible set is typically very sparse and generative filters
can be used to enumerate all the possible system designs \cite{zeidner2}.

In generative filters, one starts by defining the functional subsystems and then listing their interconnection
rules. Based on these rules one can efficiently identify all possible architectures \cite{zeidner2}. Using
generative filtering on the fixed wing aircraft system gives $27,225$ feasible architectures (significantly less
than the $2^{42}$ possible combinations of subsystem interconnection). One can now analyze and rank the resulting
architectures based on complexity and interdependencies.

After analyzing 27,225 configurations of a system, we show the most complex one and the least complex one, from
the definition of metrics in equation~(\ref{eq2}) and equation~(\ref{eq3}) with $W=\infty$. We compare results obtained by
using our spectral complexity with those obtained by using graph energy.

We compare our clustering results, with those obtained by using the Fiedler method, Cheeger bounds~\cite{chung}, and modularity maximization~\cite{communitydirected}. Our approach for the Fiedler method is as follows: at first for a given graph we construct the adjacency matrix $M$ according to equation (\ref{eq5}). Then we symmetrize the obtained matrix as $M^{(sym)}_{ij}=M_{ij}{\vee}M_{ji}$, where ${\vee}$ is the logical OR operator. After that we find the Laplacian matrix $L=D-M^{(sym)}$, where $D$ is the degree matrix. In this matrix, rows sum to zero. The Fiedler approach is based on the second smallest eigenvalue and the corresponding eigenvector of the symmetric matrix $L$. In particular, the signs of the components of the corresponding eigenvector are used to partition the graph in two parts.

In the following, N1 will correspond to the \emph{engine}, N2 to the \emph{fuel system}, N3 to the \emph{EPS}, N4 to the \emph{ECS}, N5 to the \emph{APU}, N6 to the \emph{ram cooler}, and N7 to the \emph{actuation system}.

\subsubsection{High Complexity Architecture}

After analyzing all 27,255 configurations, the architecture number $26,940$ in Fig.~\ref{conf26940} was found to be the most complex. The eigenvalues for the graph are displayed in Fig. \ref{eig26940}.

\begin{figure}[!h]
\begin{center}
\includegraphics[scale=1]{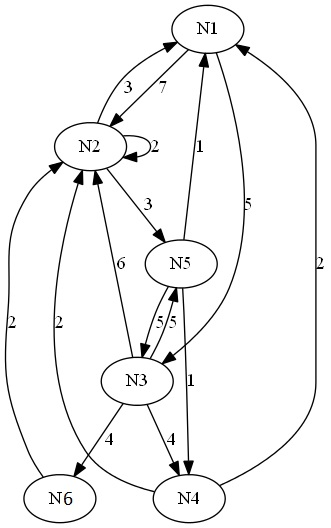}
\caption{Graph configuration 26,940. Edge weights are shown next to the edges. Node 1 has weight 20, Node 2 has
weight 8, Node 3 has weight 10, Node 4 has weight 10, Node 5 has weight 15, Node 6 has weight
4.}\label{conf26940}
\end{center}
\end{figure}

\begin{figure}[!h]
\begin{center}
\includegraphics[scale=0.5]{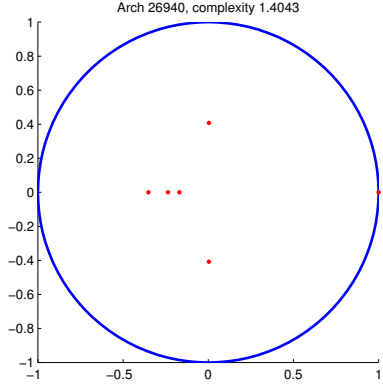}
\end{center}
\caption{\label{eig26940} Eigenvalues for high complexity architecture.}
\end{figure}

The complexity for this graph by using equation (\ref{eq2}) and $W=\infty$ in equation (\ref{eq3}) is equal to 1.4043. The
complexity for the random graph with the same number of nodes and average degree by using equation (\ref{eq2}) and
$W=\infty$ in equation (\ref{eq3}) is equal to 0.9237. The complexity predicted by equation (\ref{eq2}) for the high
complexity graph is about 152\% of the value of complexity predicted in expectation by the same equation for a random graph.
The complexity for this graph by using equation (\ref{eq2}) and $W=1$ in equation (\ref{eq3}) is equal to 1.2012.

The above complexity can be motivated from a ``system cycle'' standpoint. In particular, in Fig.~\ref{conf26940} the cycles are,
\begin{enumerate}
\item Fuel System $\rightarrow$ Fuel System (self-loop)
\item Engine $\rightarrow$ Fuel System $\rightarrow$ Engine
\item Engine $\rightarrow$ Fuel System $\rightarrow$ APU $\rightarrow$ Engine
\item Fuel System $\rightarrow$ APU $\rightarrow$ EPS $\rightarrow$ Fuel System
\item Fuel System $\rightarrow$ APU $\rightarrow$ EPS $\rightarrow$ Ram Cooler $\rightarrow$ Fuel System
\item Fuel System $\rightarrow$ APU $\rightarrow$ EPS $\rightarrow$ ECS $\rightarrow$ Engine.
\end{enumerate}
These cycles capture energy, fuel, and data flows and interactions. We note that increased interactions amongst aircraft subsystems can be related to reduced efficiencies and failures~\cite{rosero2007moving}. Thus, multiple intersecting cycles with several nodes give rise to higher complexity systems since failure in single subsystems would propagate through the cycles and across thereby requiring additional redundancies for safety.

The nodes form the following clusters: cluster 1 contains Nodes 1 (engine), 4 (ECS), and 6 (ram cooler); cluster 2 is Node 2 (fuel system), cluster 3 is
Node 3 (EPS); and cluster 4 is Node 5 (APU). Here we note that the single node clusters are ones that co-occur in multiple cycles. By visual inspection one can see the ``leaky'' (in the sense that eigenvalues corresponding to it are at a large distance from the unit circle) 4-cycle composed of the clusters; the system cycles through the 4-cycle giving rise to high complexity. This leakiness naturally arises due to the interactions of the various cycles (enumerated above) at common nodes such as Fuel System, APU etc.

The energy for this graph by using equation (\ref{eq6}) is equal to 28.3401 (sum of singular values is equal to
7.9352). If the matrix is symmetrized, then the energy for this graph by using equation (\ref{eq6}) is equal to 33.9041
(sum of singular values is equal to 9.4931). 

By using the Fiedler method the graph is divided into the following clusters: cluster 1 contains Nodes 2 (fuel system), 3 (EPS) and 6 (ram cooler);
cluster 2 contains Nodes 1 (engine), 4 (ECS) and 5 (APU), which neither captures strongly connected components nor critical nodes that co-occur in multiple cycles.

Using a Cheeger bound approach~\cite{chung}, we find that the above graph is split into two groups. Cluster 1 contains nodes  $\left[1,2\right]$ and cluster 2 contains nodes $\left[3 , 4,  5 , 6\right]$. The spectral approach for modularity maximization (by analyzing the leading eigenvector) yields a clustering where nodes $\left[1,2,4,6\right]$ are in the first cluster and nodes $\left[3,5\right]$ lie in cluster 2. Neither of these methods capture the visually evident cycling behavior. We now contrast this architecture with one of low complexity as identified by our approach.

\subsubsection{Low Complexity Architecture}

After analyzing all 27,255 configurations as above, the architecture number $1,160$ in Fig.~\ref{conf1160} was found to be the least complex, not
counting very simple 
graphs containing mostly disjoint nodes after removing sources. The eigenvalues for the graph are displayed in Fig. \ref{eig1160}.

\begin{figure}[!h]
\begin{center}
\includegraphics[scale=0.45]{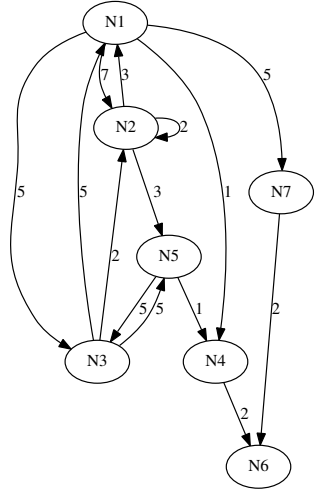}
\caption{Graph configuration 1,160. Edge weights are shown next to the edges. Node 1 has weight 20, Node 2 has
weight 8, Node 3 has weight 10, Node 4 has weight 10, Node 5 has weight 15, Node 6 has weight 4, Node 7 has weight
8.}\label{conf1160}
\end{center}
\end{figure}

\begin{figure}[!h]
\begin{center}
\includegraphics[scale=0.5]{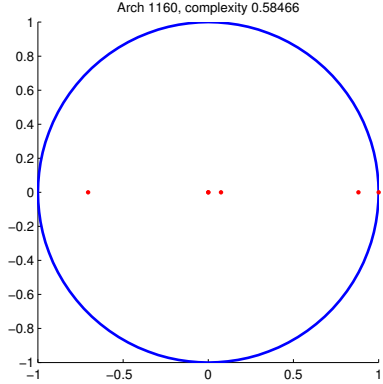}
\end{center}
\caption{\label{eig1160} Eigenvalues for low complexity architecture.}
\end{figure}

The complexity by using equation (\ref{eq2}) and $W=\infty$ in equation (\ref{eq3}) is equal to 0.5847. The complexity for
the random graph with the same number of nodes and average degree by using equation (\ref{eq2}) and $W=\infty$ in equation
(\ref{eq3}) is equal to 0.8136. The complexity predicted by equation (\ref{eq2}) for the low complexity graph is about
71\% of the value of complexity predicted in expectation by the same equation for a random graph.
The complexity by using equation (\ref{eq2}) and $W=1$ in equation (\ref{eq3}) is equal to 0.8195.

As in the previous case, the complexity can again be motivated from a ``system cycle'' standpoint. In particular, in Fig.~\ref{conf26940} the cycles are,
\begin{enumerate}
\item Fuel System $\rightarrow$ Fuel System (self-loop)
\item Engine $\rightarrow$ Fuel System $\rightarrow$ Engine
\item Engine $\rightarrow$ EPS $\rightarrow$ Engine
\item APU $\rightarrow$ EPS $\rightarrow$ APU
\item Fuel System $\rightarrow$ APU $\rightarrow$ EPS $\rightarrow$ Fuel System.
\end{enumerate}
Compared to the architecture with higher complexity, we see that this example has only $5$ cycles versus $6$ in the previous one. Additionally, the cycles in the higher complexity architecture have more nodes (hops) when compared to the low complexity architecture. Thus, the previous architecture had a higher complexity when compared to the current one despite the fact that the current example has one additional node (7 nodes) when compared to the previous one (6 nodes).

The nodes form the following clusters: cluster 1 contains Nodes 1 (engine) and 5 (APU); cluster 2 contains Nodes 2 (fuel system) and 3 (EPS). It is easy to check that these nodes generate the cycles in the graph. The
eigenvalues indicate a ``leaky''  two cycle with these two clusters. Nodes 4 (ECS), 6 (ram cooler) and 7 (actuation systems) are sinks. These
unidirectional connections lower the complexity of the system.

The energy for this graph by using equation (\ref{eq6}) is equal to 25.6040 (sum of SVDs is equal to 7.2359). If the
matrix is symmetrized, then the energy for this graph by using equation (\ref{eq6}) is equal to 34.8340 (sum of SVDs is
equal to 9.8444).
Thus, in contrast to spectral complexity, they are not much different in values obtained for the high complexity architecture. Note that the self loop of node 2 is not included in the energy
calculation.

Using a Cheeger bound approach~\cite{chung},  we find that the clustering approach finds no partition. The spectral approach for modularity maximization (by analyzing the leading eigenvector) and Fiedler method both yield a clustering where nodes $\left[1,2,3,5\right]$ are in the first cluster and nodes $\left[4,6,7\right]$ lie in cluster $2$. Once again, these methods do not capture the cycling behavior.

\subsection{Large Network Examples}
\label{largenet}
In this subsection we provide examples of calculating complexity and clustering for some large graphs.

\subsubsection{Wikipedia who-votes-on-whom network}

At first we consider the Wikipedia who-votes-on-whom network with $7,115$ nodes (\cite{web}). Nodes in the network
represent Wikipedia users and a directed edge from node $i$ to node $j$ represents that user $i$ voted for user
$j$. After removing sources, the network has 2,372 nodes. This is to be expected since most nodes are simply
voters that do not compete in elections (making them sources with no incoming edges). In Figure~\ref{wiki_adj}, we
show nonzero elements of the recurrence matrix. The multiplicity of $\lambda_{i}=0$ is $82$ and the multiplicity
of $\lambda_{i}=1$ is $1005$, which corresponds to 42.4\% of the total number of nodes. In Figure~\ref{wiki_eig},
we show all non-zero eigenvalues of the recurrence matrix.

\begin{figure}[!h]
\begin{center}
\includegraphics[scale=0.45,bb=0 0 560 420]{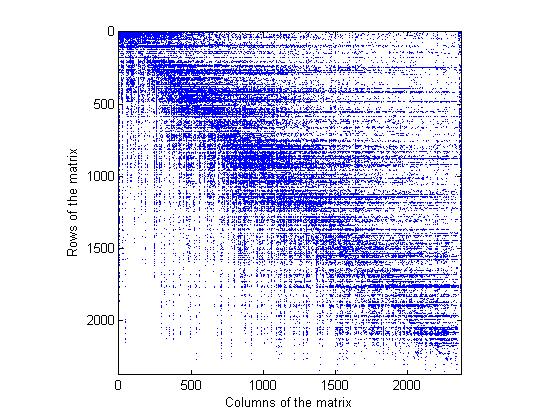}
\caption{Non-zero elements of adjacency matrix for Wikipedia who-votes-on-whom network after removing sources. The
number of non-zero elements of adjacency matrix is 57,650.}\label{wiki_adj}
\end{center}
\end{figure}

\begin{figure}[!h]
\begin{center}
\includegraphics[scale=0.5]{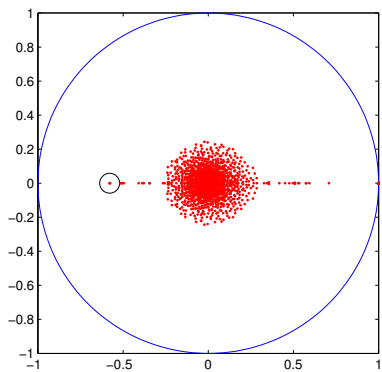}
\end{center}
\caption{\label{wiki_eig} Nonzero eigenvalues for Wikipedia who-votes-on-whom network after removing sources.}
\end{figure}

\noindent \textbf{Complexity}

\smallskip

\noindent The complexity obtained from equation (\ref{eq2}) is equal to 1.0418 (0.4938+0.5480). The complexity for the
random graph with the same number of nodes and average degree by using equation (\ref{eq2}) is equal to 1.8171
(0.8215+0.9956). The complexity predicted by equation (\ref{eq2}) for the Wikipedia who-votes-on-whom graph is about
57\% of the value of complexity predicted by the same equation for the random graph,
indicating an internal structure to the graph. Looking at the eigenvalue distribution shown in figure \ref{wiki_eig}, we see that it has the structure of randomly distributed eigenvalues inside a disk of small radius. We know from theorem \ref{the:3} that such distributions of eigenvalues yield high spectral complexity. There is also a set of eigenvalues away from that disk on positive and negative real line inside the unit disc. We next show, using clustering, that there is
internal structure corresponding to a low period -- namely period 2-cycle that contributes to an {\it eigenvalue on the negative real line that lowers complexity} over the maximally complex graph, or even a random graph.

\smallskip

\noindent \textbf{Clustering}

\smallskip

\noindent There are $56$ disjoint single nodes for Wikipedia who-votes-on-whom network which are not considered for clustering. The graph contains 1,016
sinks. The clustering was done for the strongly connected component. We obtained cluster  C1 of $622$ nodes and
cluster  C2 of $678$ nodes. The generating eigenvalue is -0.5792588, indicating a 2-cycle.


In Figure \ref{wiki_frac}, we plot the ratio of the number of edges going from cluster X to cluster Y to the number
of edges inside cluster X depending on the percentage of nodes in all clusters. The percentage of nodes in all
clusters is calculated as follows: we first sort the generating eigenvector in the ascending order. We then compute the fraction of nodes to keep such that the sum of the ratios is the maximum.

\begin{figure}[!h]
\begin{center}
\includegraphics[scale=0.6]{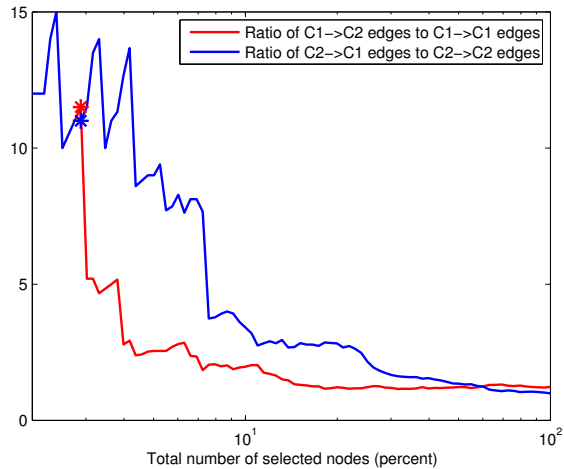}
\end{center}
\vspace{-0.5cm}
\caption{\label{wiki_frac} The ratio of the number of edges going from cluster X to cluster Y to the number of
edges inside cluster X depending on the percentage of nodes in all clusters for Wikipedia who-votes-on-whom
network. Cluster X can be cluster C1 or cluster C2 and cluster Y can be cluster C1 or cluster C2. The asterisk
shows the point where the sum of two ratios is the maximum.}
\end{figure}

In the following, we select such percentage of nodes in all clusters so that the sum of two ratios, plotted as
solid lines in Figure \ref{wiki_frac}, is the maximum. We mark the found point of 2.9\% with $*$ on Figure
\ref{wiki_frac}. The obtained graph is shown in Figure \ref{wiki_clust}, where nodes' numbers are numbers in the
graph before removing sources. The number of nodes in each cluster and the ratio of the number of edges between
clusters or inside the cluster to the number of nodes in the cluster are shown in Table \ref{table1}. The average
degree of this graph is $1.3243$,  calculated as the ratio of the total number of outgoing
edges from each cluster and edges inside each cluster to the total number of nodes in clusters. In the table, the
number in parenthesis shows the number of nodes in the corresponding cluster. Other numbers show the ratio of the
number of edges from X to Y to the number of nodes in X, where X can be cluster C1 and cluster C2 and Y can be
cluster C1 and cluster C2. As it can be seen from the table the biggest ratio is for C1 to C2 and for C2 to C1.

\begin{figure}[!h]
\begin{center}
\includegraphics[scale=0.3]{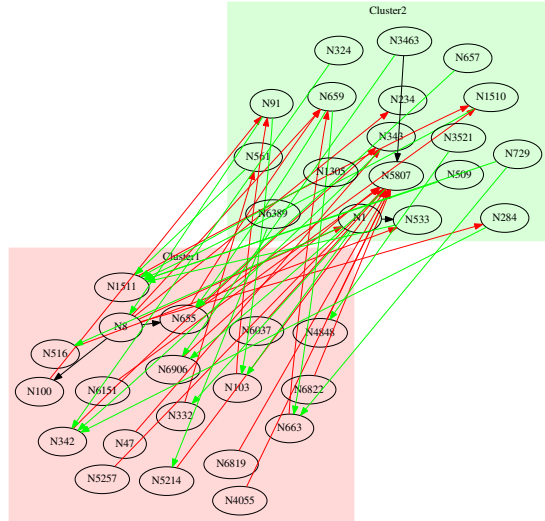}
\end{center}
\caption{\label{wiki_clust} Clustering for Wikipedia who-votes-on-whom network with 2.9\% of initial number of
nodes in both cluster C1 and cluster C2. Nodes labels are nodes numbers in the network before removing sources.
The nodes from cluster C1 are situated on light red background. The nodes from cluster C2 are situated on light
green background. The edges going from cluster C1 to cluster C2 are red, the edges going from cluster C2 to
cluster C1 are green, the edges inside clusters are black.}
\end{figure}

\begin{table}[!h]
  \centering
  \caption{The number of nodes in each cluster and the ratio of the number of edges between clusters or inside the
  cluster to the number of nodes in the cluster in Wikipedia who-votes-on-whom network with 2.9\% of initial
  number of nodes in all clusters.}\label{table1}
  \begin{tabular}{c|cc}
    X & X$\rightarrow$C1 & X$\rightarrow$C2 \\
    \hline
    C1 (19) & 0.1053 & \textbf{1.2105} \\
    C2 (18) & \textbf{1.2222} & 0.1111 \\
  \end{tabular}
\end{table}

The number of nodes in each cluster and the ratio of the number of edges between clusters or inside the cluster to
the number of nodes in the cluster in the case of 100\% of initial
  number of nodes in all clusters are shown in Table \ref{table2}. The average degree is
30.3508. As it can be seen from the table the biggest ratio is for C1 to C2.

\begin{table}[!h]
  \centering
  \caption{The number of nodes in each cluster and the ratio of the number of edges between clusters or inside the
  cluster to the number of nodes in the cluster in Wikipedia who-votes-on-whom network with 100\% of initial
  number of nodes in all clusters.}\label{table2}
  \begin{tabular}{c|cc}
    X & X$\rightarrow$C1 & X$\rightarrow$C2 \\
    \hline
    C1 (622) & 14.2910 & \textbf{17.5595} \\
    C2 (678) & 14.4130 & 14.5619 \\
  \end{tabular}
\end{table}

We also performed clustering for the strongly connected component by using the Fiedler method. We obtained cluster
1 of 1,280 nodes and cluster 2 of 20 nodes (a highly unbalanced cut). The table for the number of nodes in each cluster and the ratio of the
number of edges between clusters or inside the cluster to the number of nodes in the cluster are shown in Table
\ref{table3}. The number of edges between and inside clusters is calculated for the directed graph before the
symmetrization of the adjacency matrix. The smallest ratio is for C1 to C2, what reveals the weak connection from
C1 to C2. We see that the method is not capable of uncovering any strong internal structure in this directed graph.

\begin{table}[!h]
  \centering
  \caption{The number of nodes in each cluster and the ratio of the number of edges between clusters or inside the
  cluster to the number of nodes in the cluster in Wikipedia who-votes-on-whom network (Fiedler
  method).}\label{table3}
  \begin{tabular}{c|cc}
    X & X$\rightarrow$C1 & X$\rightarrow$C2 \\
    \hline
    C1 (1280) & 29.7891 & \textbf{0.5398} \\
    C2 (20) & 29.600 & 2.1500 \\
  \end{tabular}
\end{table}

\subsubsection{Gnutella peer to peer network}

In the following, we consider the Gnutella peer to peer network with $6,301$ nodes (\cite{web}). Nodes represent
hosts in the Gnutella network topology and edges represent connections between the Gnutella hosts. After removing
sources the network has 6,179 nodes. In Figure \ref{p2p_adj}, we show nonzero elements of the recurrence matrix.
There are 674 zero eigenvalues and 3,836 one eigenvalues, which are 62.0\% of the total number of nodes. In Figure
\ref{p2p_eig}, we show all non-zero eigenvalues of the matrix. We again see the structure similar to the Wikipedia network, but with even stronger indication of complexity indicated by the concentration of eigenvalues inside the disk of small radius. The eigenvector corresponding to the eigenvalue of
about 0.5 has zero components for sinks and the same sign nonzero components for nodes that are not sinks.

\begin{figure}[!h]
\begin{center}
\includegraphics[scale=0.45,bb=0 0 560 420]{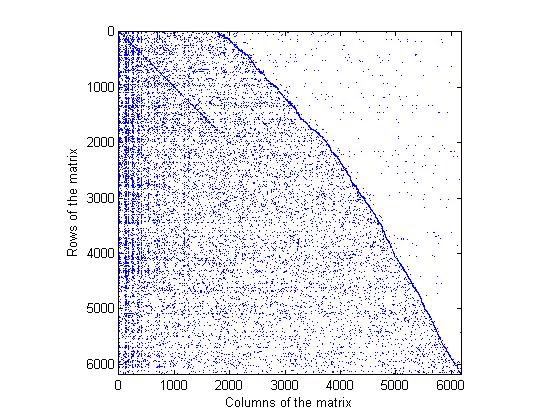}
\caption{Non-zero elements of adjacency matrix for Gnutella peer to peer network after removing sources. The
number of non-zero elements of adjacency matrix is 19744.}\label{p2p_adj}
\end{center}
\end{figure}

\begin{figure}[!h]
\begin{center}
\includegraphics[scale=0.5]{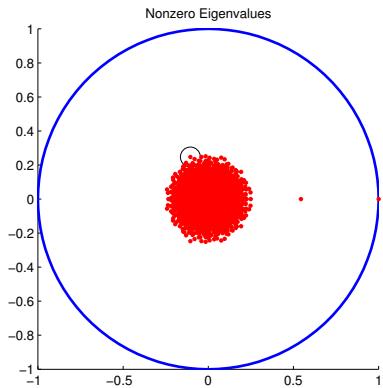}
\end{center}
\vspace{-0.5cm}
\caption{\label{p2p_eig} Nonzero eigenvalues for Gnutella peer to peer network after removing sources.}
\end{figure}

\smallskip

\noindent \textbf{Complexity}

\smallskip

\noindent The complexity by using equation (\ref{eq2}) is equal to $0.5638 (0.2661+0.2977)$. The complexity for the
random graph with the same number of nodes and average degree by using equation (\ref{eq2}) is equal to 1.5522
(0.5976+0.9546). Thus, the complexity predicted by equation (\ref{eq2}) for the Gnutella graph is about 36\% of the value of
complexity predicted by the same equation for the random graph, again indicating structure induced by a low-period cycle that we uncover next.

\smallskip

\noindent \textbf{Clustering}

\smallskip

\noindent There are $151$ disjoint single nodes in the Gnutella graph which are not considered for clustering. The graph contains 3,960
sinks. The clustering algorithm found the generating eigenvalue $-0.1054572+0.2470956i$ (see the circled eigenvalue in the figure \ref{p2p_eig}). From the associated generating eigenvector we obtained three clusters: cluster C1 of $659$ nodes, cluster C2 of $675$ nodes, cluster C3 of $734$ nodes.

In Figure \ref{p2p_frac} we plot the ratio of the number of edges going from cluster X to cluster Y to the number
of edges inside cluster X depending on the percentage of nodes in all clusters.

\begin{figure}[!h]
\begin{center}
\includegraphics[scale=0.6]{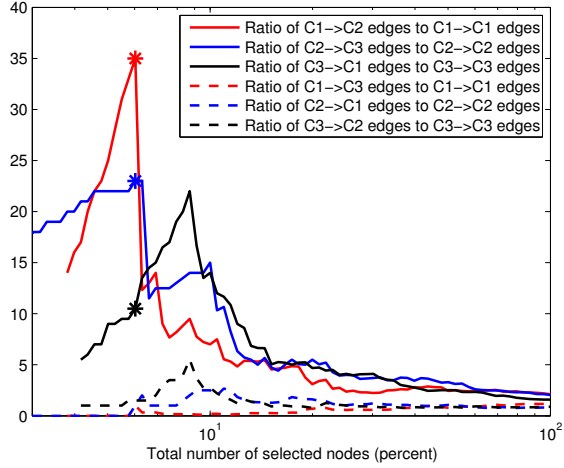}
\end{center}
\vspace{-0.5cm}
\caption{\label{p2p_frac} The ratio of the number of edges going from cluster X to cluster Y to the number of
edges inside cluster X depending on the percentage of nodes in all clusters for Gnutella network. Cluster X can be
cluster C1 or cluster C2 or cluster C3 and cluster Y can be cluster C1 or cluster C2 or cluster C3. The asterisk
shows the point where the sum of three ratios plotted as solid lines is the maximum.}
\end{figure}

In the following, we select such percentage of nodes in all clusters so that the sum of three ratios, plotted as
solid lines in Figure \ref{p2p_frac}, is the maximum. We mark the found point of 6.0\% with $*$ on Figure
\ref{p2p_frac}. After removal of nodes that become disjoint when the clusters were reduced in size this percentage
is 4.6. The obtained graph is shown in Figure \ref{p2p_clust}, where nodes' numbers are numbers in the graph
before removing sources. The number of nodes in each cluster and the ratio of the number of edges between clusters
or inside the cluster to the number of nodes in the cluster are shown in Table \ref{table4}. The average degree of
this graph is 0.9263. As it can be seen from the table the biggest ratios are for C1 $\rightarrow$ C2, C2
$\rightarrow$ C3, C3 $\rightarrow$ C1.

\begin{figure}[!h]
\begin{center}
\includegraphics[scale=0.3]{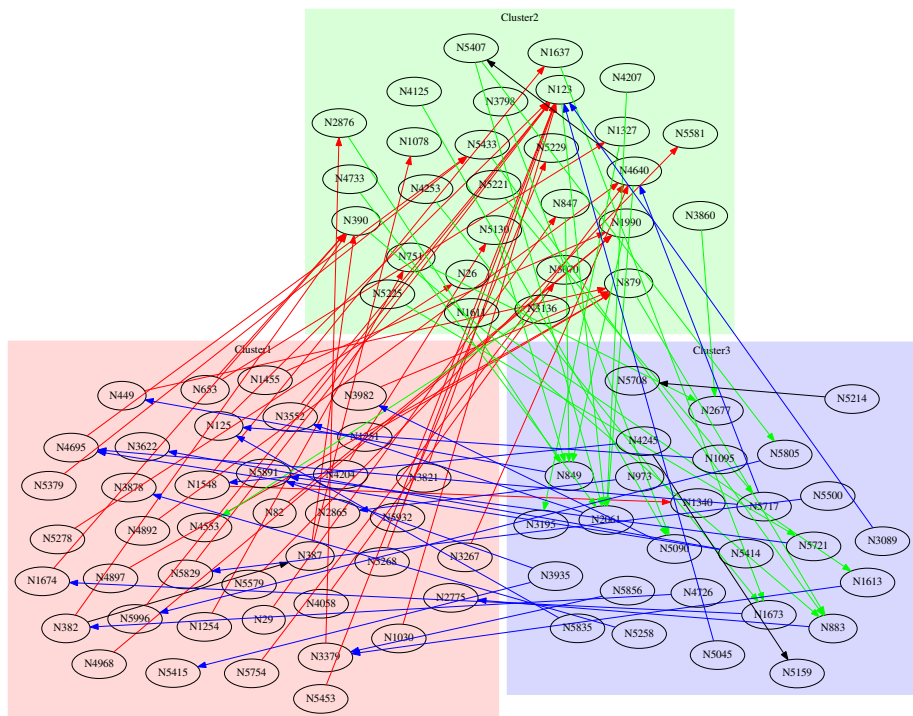}
\end{center}
\caption{\label{p2p_clust} Clustering for Gnutella peer to peer network with 4.6\% of initial number of nodes in
clusters C1, C2, C3. Nodes labels are nodes numbers in the network before removing sources. The nodes from cluster
C1 are situated on light red background. The nodes from cluster C2 are situated on light green background. The
nodes from cluster C3 are situated on light blue background. The edges going from cluster C1 are red, the edges
going from cluster C2 are green, the edges going from cluster C3 are blue, the edges inside clusters are black.}
\end{figure}

\begin{table}[!h]
  \centering
  \caption{The number of nodes in each cluster and the ratio of the number of edges between clusters or inside the
  cluster to the number of nodes in the cluster in Gnutella network with 4.6\% of initial number of nodes in all
  clusters.}\label{table4}
  \begin{tabular}{c|ccc}
    X & X$\rightarrow$C1 & X$\rightarrow$C2 & X$\rightarrow$C3\\
    \hline
    C1 (40) & 0.0250 & \textbf{0.8750} & 0.0250 \\
    C2 (28) & 0.0357 & 0.0357 & \textbf{0.8214} \\
    C3 (27) & \textbf{0.7778} & 0.1111 & 0.0741\\
  \end{tabular}
\end{table}

The number of nodes in each cluster and the ratio of the number of edges between clusters or inside the cluster to
the number of nodes in the cluster in the case of 100\% are shown in Table \ref{table5}. The average degree of
this graph is 4.5034. As it can be seen from the table the biggest ratios are for C1 $\rightarrow$ C2, C2
$\rightarrow$ C3, C3 $\rightarrow$ C1, but the ratio between them and other elements of the matrix is smaller than
in the 6\% case.

\begin{table}[!h]
  \centering
  \caption{The number of nodes in each cluster and the ratio of the number of edges between clusters or inside the
  cluster to the number of nodes in the cluster in Gnutella network with 100\% of initial number of nodes in all
  clusters.}\label{table5}
  \begin{tabular}{c|ccc}
    X & X$\rightarrow$C1 & X$\rightarrow$C2 & X$\rightarrow$C3\\
    \hline
    C1 (659) & 1.1153 & \textbf{2.3505} & 1.2337 \\
    C2 (675) & 0.9052 & 1.1037 & \textbf{2.2504} \\
    C3 (734) & \textbf{2.0763} & 1.1662 & 1.3093\\
  \end{tabular}
\end{table}

The clustering of the strongly connected component by using the Fiedler method gives cluster 1 of 1,878 nodes and
cluster 2 of 190 nodes. The number of nodes in each cluster and the ratio of the number of edges between clusters
or inside the cluster to the number of nodes in the cluster are shown in Table \ref{table6}. The number of edges
between and inside clusters is calculated for the directed graph before the symmetrization of the adjacency
matrix. As it can be seen from the table the smallest ratio is for C1 to C2, what reveals the weak connection from
C1 to C2. Again, the method fails to uncover the internal structure in the graph because the structure is of cycling type, and not of separate subgraph type.

\begin{table}[!h]
  \centering
  \caption{The number of nodes in each cluster and the ratio of the number of edges between clusters or inside the
  cluster to the number of nodes in the cluster in Gnutella network (Fiedler method).}\label{table6}
  \begin{tabular}{c|cc}
    X & X$\rightarrow$C1 & X$\rightarrow$C2 \\
    \hline
    C1 (1878) & 4.3679 & \textbf{0.2023} \\
    C2 (190) & 2.5632 & 1.2789 \\
  \end{tabular}
\end{table}

\section{Conclusions}

In this work we proposed a new, spectral, measure of complexity of systems and an associated spectral clustering
algorithm. This complexity measure (that we call \textit{spectral complexity}) is based on the spectrum of the
underlying interconnection graph of the subcomponents in the system. Spectral complexity is a natural engineering
extension to software complexity measures developed in \cite{mccabe}. We find that compared to competing
complexity measures (such as graph energy), spectral complexity is more appropriate for engineering systems. For example, one of its features is that the complexity monotonically increases with the average node degree. In addition, it properly accounts for structure and complexity features induced by cycles in a directed graph.  Using the spectral complexity measure, comparison of
complex engineered systems is enabled, potentially providing significant savings in design and testing.

Spectral complexity also provides an intuitive approach for clustering directed graphs. It partitions the graph
into subgroups that  map into one another. Our partitioning shows a strong cycling structure even for complex
networks such as Wikipedia and Gnutella, that the standard methodologies like the Fiedler vector partitioning do
not provide. 

Our methods are demonstrated on engineering systems, random graphs, Wikipedia and Gnutella examples. We find that
the high and low spectral complexity architectures uncovered by our methods correspond to an engineer's intuition of a highly complex vs. a low complexity architecture. Namely, the low complexity of the
engineered architecture is related to
more layers in its horizontal-vertical decomposition \cite{Mezic:2004,XuandLan:2015} i.e. with a graph structure
closer to acyclic.

It is of interest to note that the methods introduced here have proven to be of strong use in data-driven analysis
of dynamical systems \cite{Budisicetal:2012},
which should make it possible to combine the introduced measure of complexity with measure of dynamic complexity
for dynamical systems on networks.

\section{Acknowledgements}
We are thankful to an anonymous referee for comments that vastly  improved the presentation of the material over the original version.
The views expressed are those of the author and do not reflect the official policy or position of the Department
of Defense or the U.S. Government. Approved for Public Release, Distribution Unlimited.

This work was partially supported under AFOSR grant FA9550-17-C-0012 and under DARPA Contract FA8650-10-C-7080.

The authors thank the anonymous reviewers for their insightful feedback and suggestions.
\bibliographystyle{unsrt}
\bibliography{Complexity}
\end{document}